\iftrue
\RequirePackage{fix-cm}

\documentclass[smallextended,numbook,envcountsame]{svjour3}       % as required in the instructions + counting once in each section
\usepackage{bookmark}
\smartqed  % flush right qed marks, e.g. at end of proof

\usepackage{graphicx}

\usepackage{mathptmx}      % use Times fonts if available on your TeX system

\usepackage{amssymb}
\usepackage{amsmath}
\usepackage{a4wide}
\usepackage{amscd}
\usepackage{hyperref} %jk

\usepackage{xcolor}

 \journalname{Discrete \& Computational Geometry}

\else

\documentclass[11pt]{article}
\usepackage{geometry}                		% See geometry.pdf to learn the layout options. There are lots.
\usepackage[parfill]{parskip}    		% Activate to begin paragraphs with an empty line rather than an indent
\usepackage{graphicx}				% Use pdf, png, jpg, or eps§ with pdflatex; use eps in DVI mode
\usepackage{amssymb}
\usepackage{amsmath}
\usepackage{palatino}
\usepackage[english]{babel}
\usepackage{amscd}
\usepackage{amsthm}
\usepackage{color}
\usepackage{hyperref} %jk
\usepackage[notref,notcite]{showkeys} %jk

\fi

\newcommand\showchange[2]{\IfFileExists{option-#1-show}{\marginpar{$\bullet$ #1}$^\bullet$\footnote{#1: #2}}{}}

\iftrue
\spnewtheorem{assumption}[theorem]{Running assumption}{\bf}{\it}
\spnewtheorem{notation}[theorem]{Notation}{\bf}{\it}
\spnewtheorem*{proofC1}{Proof of Claim~1}{\it}{\rm}
\spnewtheorem*{proofC2}{Proof of Claim~2}{\it}{\rm}
\spnewtheorem*{proofC3}{Proof of Claim~3}{\it}{\rm}
\spnewtheorem{fact}[theorem]{Fact}{\bf}{\it}

\newbox\refisodiamannulus
\AtBeginDocument{\setbox\refisodiamannulus=\hbox{\ref{isodiam_annulus}}}
\spnewtheorem*{proofxx}{Proof of Theorem~\copy\refisodiamannulus}{\it}{\rm}
\else
\newtheorem{theorem}{Theorem}[section]
\newtheorem{lemma}[theorem]{Lemma}

\newtheorem{problem}[theorem]{Problem}
\newtheorem{assumption}[theorem]{Running assumption}
\newtheorem{notation}[theorem]{Notation}
\theoremstyle{definition}
\newtheorem{definition}{Definition}
\fi

\newcommand{\R}{\mathbb R}
\newcommand{\N}{\mathbb N}

\newcommand{\lambdaone}{\lambda_1}
\newcommand{\lambdatwo}{\lambda_2}
\newcommand{\polar}{{\mathrm{pol}}}   % pl was fine, I only want to hint about \polarindex macro.
\newcommand{\polarindex}{_\polar}

\newcommand{\M}{\mathcal{M}}
\newcommand{\Ha}{\mathcal{H}}
\newcommand{\TA}{\tilde{A}}
\newcommand{\X}{\boldsymbol{X}}
\newcommand{\V}{\boldsymbol{V}}
\newcommand{\x}{\boldsymbol{x}}
\newcommand{\y}{\boldsymbol{y}}
\newcommand{\Y}{\boldsymbol{Y}}
\newcommand{\BS}{\boldsymbol{S}}
\newcommand{\s}{\boldsymbol{s}}
\newcommand{\ba}{\boldsymbol{a}}

\DeclareMathOperator{\diam}{diam}

\DeclareMathOperator{\dist}{dist}
\DeclareMathOperator{\conv}{conv}
\newcommand\setsep{:\ }

\newcommand{\norm}[1]{\left\|#1\right\|}

\newcommand{\abs}[1]{\left|#1\right|}

\newcommand{\NOOPxbeforeThm}{}

\def\bred#1\ered{{\color{red}#1}}

\begin{document}
\title{A Tur\'an-type theorem for large-distance graphs in Euclidean spaces, and related isodiametric problems}

\iftrue
% \titlerunning{A Tur\'an-type theorem and isodiametric problems}
% \titlerunning{A Tur\'an-type theorem  for large-distance graphs and isodiametric problems}
%  It seems to be sufficient to remove  the comma  AND the word related.
%  So this seems to work:  \titlerunning{A Tur\'an-type theorem for large-distance graphs in Euclidean spaces and isodiametric problems}
\titlerunning{A Tur\'an-type theorem for large-distance graphs, and related isodiametric problems}

\authorrunning{
M. Dole\v{z}al,
J. Hladk\'y,
J.~Kol\'a\v{r},
T.~Mitsis,
C.~Pelekis
and
V.~Vlas\'ak
}
\author{
Martin Dole\v{z}al
\and
Jan Hladk\'y
\and
Jan~Kol\'a\v{r}
\and
Themis~Mitsis
\and
Christos~Pelekis
\and
V\'aclav~Vlas\'ak
}
\institute{
Martin Dole\v{z}al
\at
Institute of Mathematics, Czech Academy of Sciences, \v{Z}itn\'a 25, 
115 67,   Praha 1, 
Czech Republic.  
Research supported by the GA\v{C}R project 17-27844S and RVO: 67985840.
\email{dolezal@math.cas.cz}
\and
Jan Hladk\'y
\at
Institute of Mathematics, Czech Academy of Sciences, \v{Z}itn\'a 25, Praha 1, Czech Republic. Research supported by  GA\v{C}R project 18-01472Y  and RVO: 67985840. \email{honzahladky@gmail.com} 
\and
Jan~Kol\'a\v{r}
\at
        Institute of Mathematics, Czech Academy of Sciences, \v{Z}itn\'a 25, Praha 1, Czech Republic.
                    % The <xxx-named> author acknowledges the support of the EPSRC grant EP/N027531/1
                Research supported by the EPSRC grant EP/N027531/1 and by RVO: 67985840.
                \email{kolar@math.cas.cz}
\and
Themis~Mitsis
\at
Department of Mathematics and Applied Mathematics, University of Crete, 70013 Heraklion, Greece. 
\email{themis.mitsis@gmail.com}
\and
Christos~Pelekis
\at
Institute of Mathematics, Czech Academy of Sciences, \v{Z}itn\'a 25, Praha 1, Czech Republic. Research supported by the Czech Science Foundation, grant number GJ16-07822Y, by GA\v{C}R project 18-01472Y  and RVO: 67985840. \email{pelekis.chr@gmail.com}
\and
V\'aclav~Vlas\'ak
\at
Faculty of Mathematics and Physics, Charles University, Sokolovsk\'a 83, 18675 Praha 8, Czech Republic. \email{vlasakvv@gmail.com}
}
\else
\author{
Martin Dole\v{z}al \thanks{Institute of Mathematics, Czech Academy of Sciences, \v{Z}itn\'a 25, 
115 67,   Praha 1, 
Czech Republic.  
Research supported by the GA\v{C}R project 17-27844S and RVO: 67985840.
E-mail: dolezal@math.cas.cz}
\and
Jan Hladk\'y \thanks{Institute of Mathematics, Czech Academy of Sciences, \v{Z}itn\'a 25, Praha 1, Czech Republic. Research supported by  GA\v{C}R project 18-01472Y  and RVO: 67985840. E-mail: honzahladky@gmail.com} 
\and
Jan~Kol\'a\v{r}
        \thanks{Institute of Mathematics, Czech Academy of Sciences, \v{Z}itn\'a 25, Praha 1, Czech Republic.
                    % The <xxx-named> author acknowledges the support of the EPSRC grant EP/N027531/1
                Research supported by the EPSRC grant EP/N027531/1 and by RVO: 67985840.
                E-mail: kolar@math.cas.cz}
\and
Themis~Mitsis\thanks{Department of Mathematics and Applied Mathematics, University of Crete, 70013 Heraklion, Greece. 
E-mail: themis.mitsis@gmail.com}
\and
Christos~Pelekis\thanks{Institute of Mathematics, Czech Academy of Sciences, \v{Z}itn\'a 25, Praha 1, Czech Republic. Research supported by the Czech Science Foundation, grant number GJ16-07822Y, by GA\v{C}R project 18-01472Y  and RVO: 67985840. E-mail: pelekis.chr@gmail.com}
\and
V\'aclav~Vlas\'ak\thanks{Faculty of Mathematics and Physics, Charles University, Sokolovsk\'a 83, 18675 Praha 8, Czech Republic. E-mail:  vlasakvv@gmail.com}
}
\fi

\date{January 8, 2020}
% The correct dates will be entered by the editor

\def\makeheadbox{{\hbox{}}}

\maketitle

\begin{abstract}
Given a measurable set $A\subset \mathbb R^d$ we consider the
\emph{large-distance graph} $\mathcal{G}_A$, on the ground set $A$, in which
each pair of points from $A$ whose distance is bigger than~2 forms an edge. We
consider the problems of maximizing the $2d$-dimensional Lebesgue measure of
the edge set as well as the $d$-dimensional Lebesgue measure of the vertex set
of a large-distance graph in the $d$-dimensional Euclidean space that contains
no copies of a complete graph on $k$ vertices. The former problem may be seen
as a continuous analogue of Tur\'an's classical graph theorem, and the latter
as a ``graph-theoretic'' analogue of the classical isodiametric problem. Our main result yields an analogue of Mantel's theorem for large-distance graphs. Our approach employs an isodiametric inequality in an annulus, which might be of independent interest.  

%  this is  MSC (2010)
\end{abstract}

\section{Prologue, related work and main results}

Let us begin with a folklore result of Tur\'an which pertains to  graphs  
that contain no complete graph
(also called a \emph{clique})
on $k$ vertices.  Given a graph $G=(V,E)$, we denote by $|V|$ and $|E|$ the number of its vertices and edges, respectively.
                \NOOPxbeforeThm

\begin{theorem}[Tur\'an \cite{turan}]
\label{turanFund} Let $G=(V,E)$ be a graph  which does not contain a complete graph 
on $k$ vertices. Then $|E|\le  \frac{1}{2}\left(1-\frac{1}{k-1}\right) \cdot |V|^2$.  
\end{theorem}

For $k=3$, this result is due to Mantel (see~\cite{Mantel}). The extremal graph is obtained by dividing the vertex set into $k-1$ pairwise disjoint subsets whose sizes are as equal as possible, and by joining two vertices with an edge if and only if they belong to different subsets. In other words, the extremal graph is the complete balanced $(k-1)$-partite graph. 
Tur\'an's theorem is a fundamental result in extremal graph theory that has been generalised in a plethora of ways (see \cite{BollobasExtremal,SimonovitsSos}). 

In this paper, we consider similar extremal questions for measure graphs whose edge set definition is based on distances in an Euclidean space.

\subsection{Large-distance graphs}
Already in the 1970's it was realized (see \cite{Bollobas,katonaone,katonatwo,Katona}) that several results from extremal graph theory have measure-theoretic counterparts. In this setting, one is interested in the maximum ``number of edges'' in \emph{measure graphs}, i.e., graphs whose vertex set corresponds to some measure space $X$ and whose edge set corresponds to a symmetric subset of the product space that does not intersect the diagonal, i.e., it contains no points of the form $(x,x)$, where $x\in X$. This line of research received a substantial impetus recently due to the emergence of the theory of limits of dense graphs (see~\cite{Lovasz}). In this work we look at a special case of measure graphs whose vertex-set is formed by the points of a measurable set in the Euclidean space, and whose edge-set is formed by pairs of vertices that are at sufficiently large distance. 
In order to be more precise we need to introduce some  piece of notation. 

Here and later, $\lambda_d(\cdot)$ denotes $d$-dimensional Lebesgue measure.  Given a point $p\in \R^d$ and a positive 
real $r$, we denote by $D(p,r)$ the set of points whose distance from $p$ is less than 
or equal to $r$, and by $D(p,r)^o$ its interior. The boundary of $D(p,r)$ will be denoted as $S(p,r)$. 
If $r=1$ we will refer to $D(p,1)$ as a \emph{unit ball}, or a \emph{unit $d$-ball} if we want to emphasize the dimension of the underlying Euclidean space. The $d$-dimensional 
Lebesgue measure of $D(p,1)$ is denoted $\omega_d$; recall that $\omega_{d}=\frac{\pi^{d/2}}{\Gamma(d/2+1)}\;.$
The Euclidean distance between two points $p,q\in \R^d$ is denoted $\|p-q\|$ and 
the distance between two sets $A,B\subset \R^d$ is defined as 
\[ \text{dist}(A,B) = \inf\{\|a-b\|: a\in A, b\in B\} \, . \]
We shall be interested in the ``maximum number'' of edges in graphs that are associated to the distance set of measurable  subsets of the Euclidean space, and are defined as follows.
                \NOOPxbeforeThm

\begin{definition}[Large-distance graphs] 
Let $A$ be a measurable subset of $\R^d$. The \emph{large-distance graph} corresponding to $A$,  denoted $\mathcal{G}_A$,  is defined as follows: 
The vertex set of $\mathcal{G}_A$ is the set $A$. The edge set of $\mathcal{G}_A$, denoted $\mathcal{E}_A$,  is defined as 
\[ \mathcal{E}_A = \{ (p,q)\in A\times A : \norm{ p - q } > 2  \} . \]
The ``number of edges'' in $\mathcal{G}_A$ is defined as $e(\mathcal{G}_A) := \frac{1}{2}\cdot \lambda_{2d}\left(\mathcal{E}_A \right)$.
\end{definition}

Let us remark that the reason for setting the distance threshold to~2 is that our main results have a particularly nice formulation. On the other hand, only trivial rescaling would have to be introduced in order to change the distance threshold to any other positive number. 

To the best of our knowledge, large-distance graphs corresponding to measurable subsets of the Euclidean space have not been systematically studied so far. 
There exists a vast amount of literature on the, so-called, \emph{distance graphs}, where two points are joined with an edge if their distance is equal to $1$ (see, for example, \cite{ShabanovRaigorodskii}), but most research is mainly driven by the well-known problem regarding the chromatic number of the plane. A rather similar graph, whose vertex set corresponds to points of the $d$-dimensional sphere and whose edge set is formed by pairs of points that are at sufficiently large Euclidean distance, is the so-called \emph{Borsuk graph} (see~\cite[p. 30]{Matousek}, or~\cite{KahleFigueroa}). However, most of the related research so far appears to focus on the chromatic number of Borsuk graphs. In this article we focus on the problem of maximizing the ``amount'' of vertices as well as the ``amount'' of edges in large-distance graphs subject to the constraint that they do not contain a copy of a fixed finite graph. The main object of our study in large-distance graphs is introduced in the following.
                \NOOPxbeforeThm

\begin{definition}[$H$-free large-distance graphs] 
Suppose that $A\subset \R^d$ is measurable and let $H= (V(H), E(H))$ be a finite, simple, graph on $k\ge 2$ labeled  vertices, whose labels are represented by the set $\{1,\ldots,k\}$. Let $\mathcal{G}_A$ be the large-distance graph corresponding to $A$, and let $\mathcal{G}_A\langle H\rangle\subset A^k$ be 
the set consisting of all $k$-tuples, $(p_1,\ldots,p_k)$, of points in $A$ 
for which $(p_i,p_j)\in \mathcal{E}_A$ whenever $ij\in E(H)$. We say that $\mathcal{G}_A$ is \emph{$H$-free} if $\mathcal{G}_A\langle H\rangle = \emptyset$. 
We say that $\mathcal{G}_A$ is \emph{essentially $H$-free} if $\lambda_{k d}(\mathcal{G}_A\langle H\rangle) = 0$. 
%Finally, we say that $\mathcal{G}_A$ is 
%\emph{$H$-positive} if $\lambda_{k d}(\mathcal{G}_A\langle %H\rangle) > 0$.
\end{definition}

Of course, the most basic case to investigate is the case of $K_k$-free large-distance graphs, where $K_k$ denotes a clique of order $k$. We shall be interested in the problem of maximizing the $d$-dimensional Lebesgue measure of the vertex set, as well as in the problem of maximizing the $2d$-dimensional Lebesgue measure of the edge set of a $K_k$-free large-distance graph corresponding to a measurable subset of $\R^d$. 
The former problem has been considered in \cite{pelekis}, where some bounds are obtained in the $2$-dimensional case. The latter has been considered in~\cite{Bollobas}, where a Tur\'an-type theorem is  obtained in the setting of measure graphs.
In this article we further investigate both problems and we additionally demonstrate that they are interrelated.     

Let us now state a counterpart of Tur\'an's theorem for large-distance graphs.
                \NOOPxbeforeThm
\begin{theorem}
\label{cauchy}
Fix a positive integer $k\geq 2$.  
Suppose that $A\subset \R^d$ is a 
measurable set such that $\mathcal{G}_A$ is essentially $K_k$-free.

Then 
\begin{equation}\label{eq:cauchy}
        e(\mathcal{G}_A) \leq \frac{1}{2}\cdot\left( 1- \frac{1}{k-1} \right) \cdot \lambda_d(A)^2
        .
\end{equation}
Moreover, the inequality becomes an equality if and only if there exists $c\in [0,\omega_d]$ such that $A$ is a disjoint union of
a set of $\lambda_d$-measure zero and
$k-1$ disjoint sets $A_1,\ldots,A_{k-1}$ whose pairwise distances are at least $2$ and satisfy $\diam(A_i)\le 2$ and $\lambda_d(A_i)=c$, for all $i=1,\ldots,k-1$. 
\end{theorem}
Theorem~\ref{cauchy} is not really new. The bound~\eqref{eq:cauchy} follows from an old result of Bollob\'as, \cite[Theorem~1]{Bollobas}. In Section~\ref{moonMoser}, we deduce it, including the moreover part, using the formalism of graphons. 

\begin{remark}\label{remIsodiametric}
Note that the constant $c$ in Theorem~\ref{cauchy} cannot be arbitrary. Indeed the classical Isodiametric inequality (see e.g.~\cite[Theorem~11.2.1]{Burago_Zalgaller}) asserts that if $X\subset \R^d$ is a measurable set of diameter at most~2 then 
$\lambda_d(X) \le  \omega_d$. The equality is attained only when $X$ is a unit ball (modulo a nullset).
\end{remark}

At this point it may seem that the setting of large-distance graphs does not provide anything new: the bound in the above theorem follows in a rather straightforward way from a general theory of measure graphs (and from the recent progress on graph limits), and there are obvious constructions showing that these general bounds are tight even within the class of large-distance graphs. There is one important difference, though. Tur\'an's theorem is scale-free; for example knowing the optimal construction of a 100-vertex triangle-free graph allows us also to construct the optimal 100000-vertex triangle-free graph. However, this is not the case for large-distance graphs as we saw in Remark~\ref{remIsodiametric}. So, we ask for absolute bounds on the measure of vertices and edges an $K_k$-free large-distance graph may have. We call this the Clique-isodiametric problem.
                \NOOPxbeforeThm

\begin{problem}[Clique-isodiametric problem]
\label{gen_isodiam} 
Suppose that we are given $k,d\in\mathbb{N}$. Find
\[\mathfrak{V}_{d,k}=\sup_A \lambda_d(A)\quad\mbox{and}\quad \mathfrak{E}_{d,k}=\sup_A e(\mathcal{G}_A)\;,\]
where the suprema range over all measurable $K_k$-free sets $A\subset \mathbb{R}^d$.
\end{problem}

It is not difficult to see that $\mathfrak{V}_{d,k}\le (k-1) 2^{d}\omega_d$ and $\mathfrak{E}_{d,k}\le (k-1)^2 2^{2d-1}\omega^2_d$, and in particular are finite. Indeed, a $K_k$-free set $A$ cannot contain more than $k-1$ points $\{ p_i \}$ that satisfy $\norm{ p_i - p_j } > 2$,  whenever $i\neq j$.
If $\{ p_1, \dots, p_m \} \subset A$, where $m\le k-1$, is a maximal family of points with this property then $A\subset \bigcup_{i=1}^mD(p_i,2)$. Hence, by the Isodiametric inequality, $\lambda_d(A)\le (k-1)2^d\omega_d$. Consequently, $e(\mathcal{G}_A)\le (k-1)^2 2^{2d-1}\omega^2_d$ (and an improvement by a factor of $1-\frac{1}{k-1}$ follows from Theorem~\ref{cauchy}).

Notice that the first part of Problem~\ref{gen_isodiam} makes sense even for $k=2$, i.e., when $\mathcal{G}_A$ has no edges. Note that in this case Problem~\ref{gen_isodiam} is equivalent to the isodiametric problem.

We are unable to solve Problem~\ref{gen_isodiam} in general, and  focus on the case
$d=2$ and $k=3$. For this particular choice of parameters, we obtain the following $2$-dimensional analogue of Mantel's theorem.
                \NOOPxbeforeThm

\begin{theorem}\label{2d_mantel}
Let $A\subset \R^2$ be a measurable set for which $\mathcal{G}_A$ is $K_3$-free. Then $\lambdatwo(A) \le 2\pi$ as well as 
$e(\mathcal{G}_A) \le \pi^2$. 

Moreover, each set attaining the first bound is (up to changes on a set of $\lambda_2$-measure zero) a union of two non-overlapping unit balls. Each set attaining the second bound is (up to changes on a set of $\lambda_2$-measure zero) a disjoint union of two unit balls whose distance is larger than or equal to $2$.
\end{theorem}

In other words, the large-distance  graph of an 
optimal $K_3$-free set in $\R^2$ is 
a complete, $\lambdatwo$-balanced bipartite graph. The proof of Theorem~\ref{2d_mantel} is based on the following result, which may be of independent interest.

                \NOOPxbeforeThm

\begin{theorem}[Isodiametric inequality on the annulus] 
\label{isodiam_annulus}
Let $R \in [2\sqrt{2}, 4]$ and consider the  set $D:=D(0,R)\setminus D(0,2)^o \subset \R^2$. Assume that $A$ is a measurable subset of $D$ such that $\diam(A)\leq2$. Then 
\begin{equation}\label{eq:isoAnnu}
\lambdatwo(A)\leq R^2\arcsin\left(\frac{\ba}{R}\right)-4\arcsin\left(\frac{\ba}{2}\right)+2\arccos(\ba) ,
\end{equation}
where $\ba=\sqrt{\frac{-R^4+16R^2}{8(R^2+2)}}$.
\end{theorem}
The bound~\eqref{eq:isoAnnu} in Theorem~\ref{isodiam_annulus} is optimal. An example of an optimal set is given in~Figure~\ref{fig:Annulus}. Calculations that this set attains the bound in~\eqref{eq:isoAnnu} follow from Fact~\ref{fact:areaA}.
\begin{figure}
\begin{center}
		\includegraphics[scale=0.9]{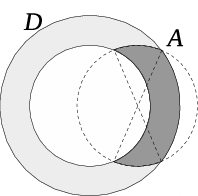}
\end{center}			
	\caption{A set $A$ showing that the bound in Theorem~\ref{isodiam_annulus} is optimal can be obtained by intersecting $D$ with a unit disk placed in such a way that its center is in the intersection of diagonals of the ``corners'' of $A$.
        As we show later, that means that the ``corners'' form a
        rectangle with vertical side length $2a$.}
	\label{fig:Annulus}
\end{figure}

\subsection{Organisation of the paper}

In Section~\ref{moonMoser} we provide a short proof of Theorem~\ref{cauchy}. 
In Section~\ref{mantel} we prove Theorem~\ref{2d_mantel}. The proof of Theorem~\ref{2d_mantel} is based upon Theorem~\ref{isodiam_annulus}, whose proof is deferred to Section~\ref{vasek}, and relies on P\'olya's  circular symmetrization, which is employed along the perimeters of co-centric circles. Our paper ends with  Section~\ref{conclusion}, in which we collect some remarks and an open problem. 

\medskip

An extended abstract describing this work was published in the proceedings of Eurocomb~2019,~\cite{DHKMPVEurocomb}.

\section{Proof of Theorem~\ref{cauchy}}\label{moonMoser}
For our proof of Theorem~\ref{cauchy}, the formalism of graphons will be convenient. We actually need only very little, so we give a self-contained introduction. We refer the reader to~\cite{Lovasz} for a thorough treatise. Suppose that $X$ is a probability measure space. A \emph{graphon} is a symmetric measurable function $W:X^2\rightarrow[0,1]$. The density of a complete graph $K_k$ in $W$ is defined as 
\begin{equation}\label{eq:densityKk}
t(K_k,W):= \int_{x_1}\cdots \int_{x_k}\prod_{1\le i<j\le k} W(x_i,x_j) \,.
\end{equation}
A graphon $W$ is said to be \emph{$K_k$-free} if $t(K_k,W)=0$. The \emph{edge density} of $W$ is $t(K_2,W)=\int_x\int_y W(x,y)$.

Our proof of Theorem~\ref{cauchy} is a rather straightforward application of the following folklore version of Tur\'an's theorem for graphons (see Corollary~16.11 in~\cite{Lovasz}).
\begin{theorem}[Tur\'an's theorem for graphons]\label{thm:TuranGraphon}
Suppose that $k\in\mathbb{N}$, $X$ is a probability measure space, and $W:X^2\rightarrow[0,1]$ is a $K_k$-free graphon. Then the edge density of $W$ is at most $1-\frac{1}{k-1}$. Moreover, if the edge density of $W$ equals $1-\frac{1}{k-1}$, then $X$ can be partitioned into sets $X_1,\ldots,X_{k-1}$ of measure $\frac{1}{k-1}$ each, and such that $W$ restricted to $X_i\times X_j$ equals to $\mathbf{1}_{i\neq j}$ (modulo a nullset) for each $i,j\in[k-1]$.
\end{theorem}

Suppose now that $A$ in Theorem~\ref{cauchy} is given. There is nothing to prove if $\lambda_d(A)=0$, so let us assume that $\lambda_d(A)>0$. Also, recall that below Problem~\ref{gen_isodiam}, we argued that $\lambda_d(A)<\infty$. So, let $X$ be a measure space whose ground set is $A$ and whose measure is the Lebesgue measure on $A$ rescaled by $\frac{1}{\lambda_d(A)}$. Let $W(x,y)$ be~1 or~0, depending on whether $xy$ forms an edge in $\mathcal{G}_A$ or not. Since $\mathcal{G}_A$ is essentially $K_k$-free, we have that $W$ is $K_k$-free. Hence, Theorem~\ref{thm:TuranGraphon} implies that the edge density of $W$ is at most $1-\frac{1}{k-1}$. Taking into account the rescaling, we get that $e(\mathcal{G}_A) \leq \frac{1}{2}\cdot( 1- \frac{1}{k-1} ) \cdot \lambda_d(A)^2$, as was needed.

Let us now turn to the moreover part of the theorem. Suppose that we have an equality in~\eqref{eq:cauchy}. Hence, the edge density of $W$ equals $1-\frac{1}{k-1}$, and hence we can partition $X$ into sets $X_1,\ldots,X_{k-1}$ as in Theorem~\ref{thm:TuranGraphon}. For each $i=1,\ldots,k-1$, let $A_i\subset X_i$ be the set of points of Lebesgue density~1 in $A_i$. We claim that for each pair $(c,d)\in A_i^2$ we have $\|c-d\|\le 2$. Indeed, if there existed a single pair $(c,d)\in A_i^2$ with $\|c-d\|>2$, then we could find a set $C\subset A_i$ (consisting only of points very close to $c$) of positive measure and a set $D\subset A_i$ (consisting only of points very close to $d$) of positive measure such that $\|c'-d'\|>2$ for each $c'\in C$ and $d'\in D$. This would imply that $W$ equals~1 on a subset of $X_i\times X_i$ of positive measure, which contradicts the assertion of Theorem~\ref{thm:TuranGraphon}. By similar reasoning, we get that if $i\neq j$ and $x\in A_i$ and $y\in A_j$, then $\|x-y\|\ge 2$. This concludes the proof of the moreover part of the theorem.

\subsection{Discussion}\label{ssec:discusion}
Recall that some extremal questions for distance graphs were considered in~\cite{ShabanovRaigorodskii}. However, there seems to be fundamental difference in extremal questions for distance graphs and for large-distance graphs. Large-distance graphs structurally behave similar to dense finite graphs, that is, graphs that have quadratically many edges in the number of vertices. This is exactly the class of graphs which is suitable for the graphon approach, and the above proof of Theorem~\ref{cauchy} demonstrates this. The same type of reductions would yield optimal versions of the Erd\H{o}s--Stone Theorem, Razborov-Reiher Clique Density Theorem, and many other, for large-distance graphs.

Distance graphs, on the other hand, correspond more to sparse finite graphs. Let us point out that some constructions of sparse finite graphs in extremal graph theory (such as~\cite{snow}) can be actually seen as finitarizations of certain distance graphs. 

\section{Proof of Theorem~\ref{2d_mantel}}\label{mantel}

In this section we prove Theorem~\ref{2d_mantel}. 
We begin by showing that $\lambdatwo(A)\le 2\pi$, where $A\subset \R^2$ is a measurable set for which $\mathcal{G}_A$ is $K_3$-free. The inner regularity of Lebesgue measure implies that it is enough to prove the result under the additional assumption that $A$ is \emph{compact}. 

The proof is based upon Theorem~\ref{isodiam_annulus}, which we assume to be true throughout this section and whose proof is deferred to the next section. 
We distinguish three cases. 

Suppose first that $A$ satisfies $\diam(A)< 2\sqrt{2}$. Then the Isodiametric inequality implies $\lambdatwo(A)< 2\pi$, as desired. 

Next assume that $\diam(A)\ge 4$. Fix two points $p,q\in A$ such that $\|p-q\|=\diam(A)$ and notice that 
$A \subset D(p,2) \cup D(q,2)$. Now observe that the set $A_p :=A\cap D(p,2)$ has diameter less than or equal to $2$; indeed, if there existed two points $x,y\in A_p$ such that $\|x-y\|>2$ then they would form together with the point $q$ a triple of points all of whose pairwise distances are larger than $2$, a contradiction to the fact that $\mathcal{G}_A$ is $K_3$-free. 
Similarly, we obtain that the set $A_q :=A\cap D(q,2)$ has diameter less than or equal to $2$. 
The Isodiametric inequality yields 
\begin{equation}\label{eq:Chester_0}
\lambdatwo(A) \le \lambdatwo(A_p) + \lambdatwo(A_q) \le 2\pi \, , 
\end{equation}
as desired. Furthermore, in case of equality, $\lambdatwo(A)=2\pi$, we must have $\lambdatwo(A_p) = \lambdatwo(A_q) =\pi$. Hence, by the characterization of equality in the Isodiametric inequality, $A_p$ and $A_q$ are two unit balls, modulo a nullset. Since $\diam(A)\ge 4$, their intersection is either empty or a singleton.

Hence we are left with the case $R:=\diam(A)\in [2\sqrt{2},4]$. Fix two points $x,y\in A$ such that 
$\|x-y\|=R$, and consider the sets 
\[ A_x = A\setminus D(x,2) \; \text{ and } \; A_y = A\setminus D(y,2) . \] 
For the same reasons as above, we have that $A\subset D(x,2) \cup D(y,2)$. Hence
\begin{equation}\label{eq:Chester}
 A \subset A_x \cup A_y \cup ( D(x,2) \cap D(y,2) ) \, .
 \end{equation}
Moreover, we have $\diam(A_x), \diam(A_y) \le 2$ as well as 
\[ A_x \subset D(x,R)  \setminus D(x,2) \; \text{ and } \; A_y \subset D(y,R)\setminus D(y,2) . \]

\bigbreak

\begin{notation}
Below, we define functions $g,\ba,h,H,f:[2\sqrt{2},4]\rightarrow\mathbb{R}$. We denote their derivatives as $g',\ba',h',H',f'$, respectively. 
\begin{align*}
g(z)&=4\arccos\left(\frac{z}{4}\right)-\frac{z}{2}\sqrt{4-\frac{z^2}{4}},\\
\ba(z)&=\sqrt{\frac{-z^4+16z^2}{8(z^2+2)}},\\
h(z)&=z^2\arcsin\left(\frac{\ba(z)}{z}\right)-4\arcsin\left(\frac{\ba(z)}{2}\right)+2\arccos(\ba(z)),\\
H(z)&=z\frac{z\ba'(z)-\ba(z)}{\sqrt{z^2-(\ba(z))^2}}-\frac{4\ba'(z)}{\sqrt{4-(\ba(z))^2}}-\frac{2\ba'(z)}{\sqrt{1-(\ba(z))^2}},\\
f(z)&=2(g(z)+h(z)).
\end{align*}
\end{notation}

\bigbreak
%\medbreak

\begin{lemma}\label{dul}
We have $\lambdatwo(A)\leq f(R)$. 
\end{lemma}
\begin{proof}
It is not difficult to see that 
\[ \lambdatwo(D(x,2) \cap D(y,2)) = 2g(R) . \] 
Moreover, by Theorem~\ref{isodiam_annulus} we have 
\[ \lambdatwo(A_x) \le h(R) \; \text{ and } \; \lambdatwo(A_y)\le h(R) . \]
Plugging this into~\eqref{eq:Chester}, we get 
\[
\lambdatwo(A) = \lambdatwo(A_x) + \lambdatwo(A_y) + \lambdatwo(D(x,2) \cap D(y,2))  \le f(R) \, , 
\]
as required. 
\qed\end{proof}

Clearly, we have $f(4)=2\pi$ and Lemma~\ref{dul} implies that it is enough to show that the function 
$f(\cdot)$ is increasing on the interval $[2\sqrt{2}, 4]$. 
Now, straightforward calculations provide the following expressions for the derivatives of the functions under consideration: 
\begin{align}
\notag
g'(z)&=-\frac12\sqrt{16-z^2}\, ,\\
\label{eq:aderiva}
\frac{\ba'(z)}{\ba(z)}&=\frac{2z}{2\ba(z)^2}\,\frac{-z^4-4z^2+32}{8(z^2+2)^2}
                      =\frac{(z^2-4)(z^2+8)}{z(z^2+2)(z^2-16)} \, ,\\
\notag
\ba'(z)&=-\frac{(z^2-4)(z^2+8)}{(2(z^2+2))^{\frac32}\sqrt{16-z^2}} \, ,\\
\notag
h'(z)&=2\arcsin\left(\frac{\ba(z)}{z}\right)z+H(z) \, .
\end{align}

Next, we want to express $H(z)$. To this end, write
\begin{align*}
Q &= 8(z^2+2)\,,\\
        S_1
       &= z^2 - \ba(z)^2
        = \frac{ 8z^4 + 16 z^2 + z^4 -16z^2}{8(z^2+2)}
        = \frac{ 9 z^4 }{ Q }
        \, ,
       \\
        S_2
       &= 4 - \ba(z)^2
        = \frac{ (z^2+8)^2 }{ Q }
        \, ,
       \\
        S_3
       &= 1 - \ba(z)^2
        = \frac{ (z^2-4)^2 }{ Q }
        \, .
\end{align*}
Furthermore, let $R_i = \sqrt{ S_i \, Q \, }$, for $i=1,2,3$. Then
\[
 R_1 = 3z^2,
 \qquad
 R_2 = z^2+8\, ,
 \qquad
 R_3 = z^2-4\, .
\]
Now
\begin{align*}
    \frac{ H(z) }{ \sqrt Q }
    &=
    \left( \frac{ z^2 }{ R_1 } - \frac{ 4 }{ R_2 } - \frac{ 2 }{ R_3 } \right) \ba'(z)
    -
    \frac z {R_1} \cdot \ba(z) \\
    &=
    \frac{ \ba(z) }{ 3 z } \left(
                           \frac{ \ba'(z) }{ \ba(z) }\cdot \,\frac{ z (z^2+2)(z^2-16) }{ (z^2-4)(z^2+8) }
                           - 1
                        \right)
                        \, .
\end{align*}
Using~\eqref{eq:aderiva}, we obtain
\begin{equation}\label{eq:Hul}
H(z)=0 \, .
\end{equation}
Observe that on $z\in[2\sqrt{2},4]$, we have $0\le \frac{\ba(z)}z\le 1$. Using this together with $\arcsin(w)\ge w$ for $w\in[0,1]$, we obtain
\begin{equation}\label{eq:SSNZ}
    \arcsin\left(\frac{\ba(z)}{z}\right)z\ge \ba(z)\;.
\end{equation}
We are now ready to express $f'(z)$:
\begin{align*}
f'(z)=2(g'(z)+h'(z))&\overset{\eqref{eq:Hul}}{=} 4\arcsin\left(\frac{\ba(z)}{z}\right)z-\sqrt{16-z^2} \\ 
&\overset{\eqref{eq:SSNZ}}{\ge} 4\ba(z)-\sqrt{16-z^2}\\ 
&= z\sqrt{\frac{2(16-z^2)}{z^2+2}}-\sqrt{16-z^2}\\
&=\sqrt{\frac{16-z^2}{z^2+2}}(\sqrt{2}z-\sqrt{z^2+2}) \, .
\end{align*}
It follows that $f'(z)>0$
whenever
$z<4$ and $\sqrt{2}z-\sqrt{z^2+2}>0$, which holds true for $z\in(2\sqrt{2},4)$. As $f$ is clearly continuous on the interval $[2\sqrt{2},4]$, the first bound in the main part of Theorem~\ref{2d_mantel} follows. The second bound follows from Theorem~\ref{cauchy}. 

Let us quickly comment on the ``moreover'' part of Theorem~\ref{2d_mantel}.  
We showed that $f$ is strictly increasing on the interval $[2\sqrt{2},4]$; hence $\lambda_2(A) < 2\pi$, when $\diam(A)<4$. 
If $\diam(A)\ge 4$ then  
\eqref{eq:Chester_0} implies that  $\lambdatwo(A)\le 2\pi$. The isodiametric inequality yields the first statement of the ``moreover'' part. The second statement follows from the ``moreover'' part of Theorem~\ref{cauchy}.

\section{Proof of Theorem~\ref{isodiam_annulus}}\label{vasek}

This section is devoted to the proof of Theorem~\ref{isodiam_annulus}, i.e., we prove the isodiametric inequality in the annulus. Recall that $D$ denotes the annulus $D(0,R)\setminus D(0,2)^o$, where $R \in [2\sqrt{2}, 4]$ is fixed. 
We write $(r,\alpha)\polarindex$ to denote the element of $\R^2$ represented by the polar coordinates  $(r,\alpha)$.
That is, $(r,\alpha)\polarindex = (r \cos \alpha, r \sin \alpha)$.

The idea behind the proof is to find a ``well behaved'' and ``maximal'' subset of $D$ whose diameter is less than or equal to $2$, and then compute its Lebesgue measure. Let us begin by clarifying the meaning of the term ``maximal''.
                \NOOPxbeforeThm

\begin{definition}\label{def:MandMax}
Set  $\M=\{A\subset D: \ A\text{ is compact and } \diam(A)\leq2\}$. We say that $M\in\M$ is $D$-maximal if $\lambdatwo(M)=\sup\{\lambdatwo(A);\ A\in\M\}$.
\end{definition}

Our first result shows that $D$-maximal sets do exist. Throughout this section, $\conv(\cdot)$ denotes convex hull.
                \NOOPxbeforeThm

\begin{lemma}\label{l:max}
There exists a $D$-maximal set. 
\end{lemma}
\begin{proof}
For every positive integer $n$, there exists $K_n\in\M$ such that 
\[ \lambdatwo(K_n)>\sup\{\lambdatwo(A);\ A\in\M\}-\tfrac1n \, . \] 
Since $\M$ endowed with the
Hausdorff metric is a compact space, we can find a  convergent subsequence $\{K_{n_k}\}_{k\in \N}$.

Let $K= \lim_{k\to\infty}K_{n_k}$.
                Clearly, $K\in\M$.
                For every $\varepsilon> 0$ there exists an open set $G\supset K$
                such that
                $\lambdatwo(G)\le \lambdatwo(K) + \varepsilon$.
                % for every $\varepsilon>0$ there exists $k_0$ such that for every $k>k_0$ we have
                For all $k$ sufficiently large, we have $K_{n_k} \subset G$, so that
                $\lambdatwo(K_{n_k}) - \varepsilon \le \lambdatwo(K)$.
                Since $\varepsilon>0$ was arbitrary, $\lambdatwo(K) \ge \sup\{\lambdatwo(A);\ A\in\M\}$
                and $K$ is a $D$-maximal set.
\qed\end{proof}

To prove Theorem~\ref{isodiam_annulus}, we wish to show that the inequality~\eqref{eq:isoAnnu} holds true
for every measurable subset of $D$ with diameter at most $2$. Since the closure of a set preserves its diameter, it clearly suffices to show that~\eqref{eq:isoAnnu} holds true for every $A \in \M$. By Lemma~\ref{l:max}, it is enough to show that~\eqref{eq:isoAnnu} holds true
for every $D$-maximal set $A\in \M$.
We will restrict our attention to such sets~$A$.
                \NOOPxbeforeThm

\begin{assumption}\label{St_a}
     $A\in \M$ is a $D$-maximal set.
\end{assumption}

Since $\diam(A) \le 2$
(and therefore the same is true for the ``radial projection'' of $A$ on $S(0,2)$) we may assume without loss of generality that 
\begin{equation}\label{polar_hypothesis} 
        \alpha \in (-\pi/2, \pi/2)
        \qquad
        \text{ for every } (r,\alpha)\polarindex \in A
        . 
\end{equation}

Next we recall the notion of circular symmetrization, which ``symmetrizes'' a given set along the perimeters of balls of fixed radius. This notion is due to P\'olya (see~\cite{Polya} or \cite[p. 77]{Bonnesen_Fenchel}). 
Actually, the following definition is slightly adjusted to our needs as we will symmetrize compact subsets of an annulus.
                \NOOPxbeforeThm

\begin{definition}[Circular symmetrization]
Let $A\subset D$ be compact and $r\in[2,R]$. Set
$A_{\circ r}=A \cap S(0,r)=\{x\in A;\ \|x\|=r\}$.
Define
\[
  A^{st}
  =
  \left\{
  (r,\alpha)\polarindex\in\R^2
  :\ r\in[2,R],\ A_{\circ r}\neq \emptyset \text{ and } |\alpha|\leq\frac{ \Ha^{1}(A_{\circ r})}{2r}
  \right\}
  \, , 
\]
where $\Ha^{1}(A_{\circ r})$ denotes the $1$-dimensional Hausdorff measure (or, the length) of $A_{\circ r}$.
Note that $\Ha^{1} ( A_{\circ r} ) = r \, \lambdaone ( \{ \alpha \in [0, 2\pi) \setsep  (r,\alpha)\polarindex\in  A \})$.
\end{definition}

In other words, the circular symmetrization replaces the set $A_{\circ r}$ with an arc, having the same $1$-dimensional Hausdorff measure as $A_{\circ r}$, that is centered on the $x$-axis.  
Our aim is to show that we may assume that the set $A$ is symmetric. Before doing so, we first show that the circular symmetrization of a compact set is compact.
                \NOOPxbeforeThm

\begin{lemma}\label{COMPACTNESS}
If $K\subset\ D$ is compact, then so is $K^{st}$.
\end{lemma}

\begin{proof}
Suppose that $\left((r_n,\alpha_n)\polarindex\in K^{st}\right)_n$ is a sequence converging to some $(r,\alpha)\polarindex$. We need to show that $(r,\alpha)\polarindex\in K^{st}$. For every $t>0$, we set $K_{\circ t}=K\cap S(0,t)$ and $C_t=\{\alpha\in(-\pi,\pi]\colon(t,\alpha)\polarindex\in K_{\circ t}\}\subset\mathbb{R}$. As $K_{\circ r_n}\neq\emptyset$ for every $n$, compactness of $K$ implies that $K_{\circ r}\neq\emptyset$ as well. So it clearly suffices to show that $\mathcal H^1(K_{\circ r})\ge\limsup_{n\rightarrow\infty}\mathcal H^1(K_{\circ r_n})$. Fix $\varepsilon>0$ and find an open set $G\subset\R$ such that $C_r\subset G$ and $\mathcal H^1(G)<\mathcal H^1(C_r)+\varepsilon$. 

We claim that there is a natural number $n_0$ such that for every $n\ge n_0$ we have $C_{r_n}\subset G$. Indeed, otherwise, there exists an infinite sequence $n_1<n_2<\ldots$ and for each $k\in\mathbb{N}$ we can find a $\beta_{n_k}\in C_{r_{n_k}}\setminus G$. By passing to a subsequence, we can assume that $(\beta_{n_k})_k$ is convergent. Let $\beta:=\lim_k \beta_{n_k}$. Since $G$ is open and $\beta_{n_k}\notin G$, we have $\beta\notin G$. But then the sequence $\left((r_{n_k},\beta_{n_k})\polarindex\right)_k$ converges to $(r,\beta)\polarindex\notin K$, which contradicts our assumption that $K$ is compact. Hence the existence of $n_0$ as above follows.

For every $n\ge n_0$ we have $\mathcal H^1(C_{r_n})\le\mathcal H^1(G)<\mathcal H^1(C_r)+\varepsilon$. As $\varepsilon>0$ was arbitrary, we conclude that $\mathcal H^1(C_r)\ge\limsup_{n\rightarrow\infty}\mathcal H^1(C_{r_n})$, and so
\begin{equation*}
\mathcal H^1(K_{\circ r})=r\mathcal H^1(C_r)\ge r\limsup_{n\rightarrow\infty}\mathcal H^1(C_{r_n})=\limsup_{n\rightarrow\infty}r_n\mathcal H^1(C_{r_n})=\limsup_{n\rightarrow\infty}\mathcal H^1(K_{\circ r_n})\,,
\end{equation*}
as was needed.
\qed\end{proof}

\begin{lemma}\label{STPOL}
If $A$ is $D$-maximal then $A^{st}$ is also $D$-maximal.
\end{lemma}
\begin{proof}
Clearly, $A^{st}\subset D$, and it is compact by Lemma~\ref{COMPACTNESS}. Moreover, Fubini's theorem yields $\lambdatwo(A)=\lambdatwo(A^{st})$ and therefore it only remains to show that $\diam(A^{st})\le 2$. To this end, we need to prove that for every 
$(r,\alpha), (t,\beta) \in [2,R] \times (-\pi,\pi]$
such that
$A_{\circ r}\neq\emptyset$,
$A_{\circ t}\neq\emptyset$,
\(
      \abs\alpha\leq\frac{ \Ha^{1}(A_{\circ r})}{2r}
\)
and
\(
      \abs\beta\leq\frac{ \Ha^{1}(A_{\circ t})}{2t}
\), we have $\|(r,\alpha)\polarindex-(t,\beta)\polarindex\|\le 2$.
Note that there exist $\alpha_1>\alpha_2$ and $\beta_1>\beta_2$ such that $(r,\alpha_i)\polarindex,(t,\beta_i)\polarindex\in A$, for $i\in\{1,2\}$, and 
\begin{equation}\label{eq:Cjb}
\alpha_1-\alpha_2\geq\frac{\Ha^{1}(A_{\circ r})}{r}\geq2|\alpha| \; \text{ and } \; \beta_1-\beta_2\geq\frac{\Ha^{1}(A_{\circ t})}{t}\geq2|\beta| \, .    
\end{equation}
We distinguish two cases. 
Assume first that $\alpha_1+\alpha_2\geq\beta_1+\beta_2$. Then we have 
\begin{equation}\label{eq:a1-b2}
|\alpha_1-\beta_2|\geq\left|\alpha_1-\frac{\alpha_1+\alpha_2}{2}\right|+\left|\frac{\beta_1+\beta_2}{2}-\beta_2\right|\overset{\eqref{eq:Cjb}}{\geq}|\alpha|+|\beta|\geq|\alpha-\beta|  
\, .
\end{equation}
Recall that $\alpha,\beta,\alpha_i,\beta_i\in\left(-\frac{\pi}{2},\frac{\pi}{2}\right)$ by~\eqref{polar_hypothesis}.
It is easy to see that $\norm{ (r,\gamma)\polarindex - (t, 0)\polarindex }$ is an increasing function of $\gamma \in [0, \pi)$.
Therefore~\eqref{eq:a1-b2}
implies that
\begin{align*}
           \norm{ (r,\alpha_1)\polarindex-(t,\beta_2)\polarindex }
      &=    \norm{ (r, \alpha_1-\beta_2)\polarindex- (t,0)\polarindex }
      \geq \norm{ (r,\alpha-\beta)\polarindex-(t,0)\polarindex }\\
      &=    \norm{ (r,\alpha)\polarindex-(t,\beta)\polarindex }\, .
\end{align*}
Hence $\norm{ (r,\alpha)\polarindex-(t,\beta)\polarindex } \le \diam A \le 2$, as required.
If $\alpha_1+\alpha_2 < \beta_1+\beta_2$ then we proceed in the same way, the only difference being that we select $\alpha_2,\beta_1$, instead of $\alpha_1,\beta_2$, and obtain the same result. 
\qed\end{proof}

Our next lemma says that we can find a $D$-maximal set which, in addition to being symmetric, also contains point $(R,0)$.

\begin{lemma}\label{R0inA}
There exists a $D$-maximal set $A$ such that $A=A^{st}$ and $(R,0)\in A$.
\end{lemma}
\begin{proof}
Let $M$ be a $D$-maximal set and let $t=\min\{v\in\R;\ M\subset D(0,v)\}$. Notice that $(t,0)\in M^{st}$. 
Define the set  
\[ A=(M^{st}+ (R-t,0))^{st} \, ,  \]
where $M^{st}+x$ denotes the set $M^{st}$ shifted by vector $x$.
By construction, and using Lemmata~\ref{COMPACTNESS} and \ref{STPOL}, we deduce $A=A^{st}$ and $A\in\M$. Moreover, the $D$-maximality of $M$ implies 
that $A$ is $D$-maximal.
Also, $(R,0) \in A$ as desired. 
\qed\end{proof}

Hence, for the remaining part of this section, we may assume that the set  $A$ satisfies the properties given by Lemma~\ref{R0inA}.
                \NOOPxbeforeThm

\begin{assumption}\label{St_a_2}
The set $A$ satisfies $A=A^{st}$ and $(R,0)\in A$. 
\end{assumption}

Note that the assumption implies that $A\subset D((R,0), 2)$.
                \NOOPxbeforeThm

\begin{lemma}
The point $(2,0)$ belongs to $A$.
\end{lemma}
\begin{proof}
Assume, towards a contradiction, that $(2,0) \notin A$ and notice that the assumption $A=A^{st}$ implies $A\cap S(0,2)=\emptyset$. We distinguish two cases.

Assume first that $A\subset D((R,0),2)^o$.
Then there exists $\varepsilon>0$ such that 
\[
\dist\bigl(A,S(0,2)\cup S((R,0),2)\bigr)>\varepsilon \, . 
\]
Consider the set $\TA=\bigl(A\cup D((R,0),\varepsilon)\bigr)-(\varepsilon,0)$. Clearly, $\lambdatwo(\TA)>\lambdatwo(A)$ and $\TA\in\M$. This implies that $A$ is not $D$-maximal, which contradicts Running assumption~\ref{St_a}.

Now, assume $A\cap S((R,0),2)\neq\emptyset$ and fix $(x_1,x_2)\in A\cap S((R,0),2)$. Recall that our Running assumption~\ref{St_a_2} implies
that
\begin{equation}\label{eq:bothpm}
  (x_1,\pm x_2)\in A\cap S((R,0),2)
  .
\end{equation}
% for both signs.

For a while, let us assume there exists a point $(y_1,y_2)\in A\setminus D((R-2,0),2)$.
Because of the symmetry in~\eqref{eq:bothpm}, we can assume $y_2 x_2 \ge 0$.        
Since $(y_1,y_2)\in A\subset D((R,0),2)$, it follows that there exists a real number $r\in[R-2,R]$ such that $(r,0)\in\conv(\{(y_1,y_2),(x_1,-x_2)\})$. The triangle inequality, gives   
\begin{align}
\begin{split}\label{eq:BorisVian}
        \norm{ (y_1,y_2)-(x_1,-x_2) } &= \norm{ (y_1,y_2)-(r,0) } + \norm{ (r,0)-(x_1,-x_2) }
\\ 
&\geq 
        \norm{ (y_1,y_2)-(R-2,0) }
        -(r - (R-2))
        +
\\
&\phantom{{}\geq{}}
        +\norm{ (R,0)-(x_1,-x_2) }
        -(R - r).
\end{split}
\end{align}
We have $\norm{ (y_1,y_2)-(R-2,0) }>2$ since $(y_1,y_2)\notin D((R-2,0),2)$. Further, we have $\norm{ (R,0)-(x_1,-x_2) }=2$ since $(x_1,x_2)\in S((R,0),2)$. 
Plugging this into~\eqref{eq:BorisVian}, we get $\norm{ (y_1,y_2)-(x_1,-x_2) }>2$ which contradicts the fact that $\diam(A)\le 2$.
So there is no such point $(y_1, y_2)$,
meaning that $A\subset D((R-2,0),2)$.
Since $\diam(A) \leq 2$ we obtain $\diam\bigl(A\cup\{(R-2,0)\}\bigr)\leq2$.
Now consider the set  
\[ \TA=\conv(A\cup\{(R-2,0)\})\cap D \, . \] 
Clearly,
$\TA\in\M$.
We have also $\lambdatwo(\TA)>\lambdatwo(A)$
because $\TA \setminus A$ contains the point $(2,0)$ and its small neighborhood (relative to $D$).
Hence $A$ is not $D$-maximal, contrary to Running assumption~\ref{St_a}. 
The result follows.  
\qed\end{proof}

\begin{notation}\label{not_1}
The following quantities will remain fixed for the remaining part of this section. 
\begin{align*}
\boldsymbol{\alpha}&=\max\{\gamma\in(-\pi/2, \pi/2);\ (R,\gamma)\polarindex\in A\}\\
\boldsymbol{\beta}&=\max\{\gamma\in(-\pi/2, \pi/2);\ (2,\gamma)\polarindex\in A\}\\
\X^1&=(\x^1_1,\x^1_2)=(R,\boldsymbol{\alpha})\polarindex\\ \X^2&=(\x^2_1,\x^2_2)=(R,-\boldsymbol{\alpha})\polarindex\\
\Y^1&=(\y^1_1,\y^1_2)=(2,\boldsymbol{\beta})\polarindex\\ \Y^2&=(\y^2_1,\y^2_2)=(2,-\boldsymbol{\beta})\polarindex
\end{align*}
Finally, let $\BS=(\s,0)$ be the point of intersection of the line segments $\conv\{\X^1,\Y^2\}$ and $\conv\{\X^2,\Y^1\}$.
\end{notation}
\bigbreak

\begin{lemma}\label{conv}
There exists a convex set $T$ such that $A=T\setminus D(0,2)^o$.
\end{lemma}
\begin{proof}
Since the diameter does not increase upon taking convex hulls, 
this immediately follows from the fact that $\TA := \conv(A)\setminus D(0,2)^o\in\M$.
Indeed, the assumption that $A$ is $D$-maximal implies that $\lambdatwo(\TA)=\lambdatwo(A)$.
It remains to observe that the set $\TA\setminus A$ is empty. Indeed, assuming that it is non empty, one can easily use Running assumption~\ref{St_a_2} to conclude that $\lambdatwo(\TA\setminus A)>0$, a contradiction.
\qed\end{proof}

The next two lemmata  are concerned with the distances 
between the points defined in Notation~\ref{not_1}.
                \NOOPxbeforeThm

\begin{lemma}\label{dlouha}
We have 
$\|\X^1-\Y^2\|=\|\X^2-\Y^1\|=2$.
\end{lemma}
\begin{proof}
As $\norm{ \X^1-\Y^2 } = \norm{ \X^2-\Y^1 }$ we only need to prove that  $\norm{ \X^1-\Y^2 } = 2$.
Assume, towards a contradiction,  that 
\begin{equation}\label{eq:stravenky}
    \|\X^1-\Y^2\|<2\;.
\end{equation} There are three cases to consider.

First, assume that $A\subset D(\X^1,2)^o$. Then there exists $\epsilon>0$ such that $A\subset D(\X^1,2-\epsilon)^o$. Set $\TA=(A\cup D(\X^1,\epsilon))\cap D$. Obviously, $\TA\in\M$. Furthermore, we claim that $\lambdatwo(\TA)>\lambdatwo(A)$. Indeed, observe first that $S(0,R)\cap \TA$ has bigger length than $S(0,R)\cap A$ by the maximality of the angle $\alpha$. Now, since both $\TA$ and $A$ are closed, we get that $\lambdatwo(\TA)>\lambdatwo(A)$. So, $A$ is not $D$-maximal, which contradicts Running assumption~\ref{St_a}.

Now, assume that $A\subset D(\Y^2,2)^o$. Then there exists $\epsilon>0$ such that $A\subset D(\Y^2,2-\epsilon)^o$. Set $\TA=(A\cup D(\Y^2,\epsilon))\cap D$. By the same arguments as above,  $\TA\in\M$ and $\lambdatwo(\TA)>\lambdatwo(A)$. This implies that $A$ is not $D$-maximal, and contradicts Running assumption~\ref{St_a}.

It remains to consider the case where there exist points $\V^1=(v_1^1,v_2^1)\in A\cap S(\Y^2,2)$ and $\V^2=(v^2_1,v^2_2)\in A\cap S(\X^1,2)$. We may assume further that $v^1_2\geq0$ and $v^2_2\leq0$; indeed, if $v^1_2<0$ then $\|(v_1^1,-v_2^1)-\Y^2\|\geq\|(v_1^1,v_2^1)-\Y^2\|=2$ and $(v_1^1,-v_2^1)\in A$. Thus we can choose the point $\V^1=(v_1^1,-v_2^1)$ instead. Similarly for $v^2_2>0$. 
Using~\eqref{eq:stravenky}, we get $\|\V^1-\Y^2\|=2>\|\X^1-\Y^2\|$.

From Lemma~\ref{conv} we easily see that $\y^2_1 \le \x^1_1$ and
% $v^1_1, v^2_1 \in [\y^2_1,\x^1_1]$.
that the real numbers $v^1_1, v^2_1$ belong to the interval $[\y^2_1,\x^1_1]$.
Since $S(\Y^2,2) \cap ([\y^2_1,\x^1_1] \times (0,\infty))$
is the graph of a decreasing function, we have $v^1_2>\x^1_2$. 
In a similar way, we get $v^2_2<\y^2_2$. 
Therefore,
the line segments $\conv(\{\V^1,\V^2\})$ and $\conv(\{\X^1,\Y^2\})$ intersect in a point, say, $Z$. The triangle inequality now yields 
\begin{align*}
         \norm { \V^1 - \V^2 }
 &=
         \norm { \V^1 - Z }
         +
         \norm { Z - \V^2 }
 \\
 &\geq
        (
         \norm { \V^1-\Y^2 }
         -
         \norm{ \Y^2 - Z }
        )
         +
        (
         \norm{ \X^1-\V^2 }
         -
         \norm{ Z - \X^1 }
        )
 \\
 &=
         \norm { \V^1-\Y^2 }
         +
         \norm{ \X^1-\V^2 }
         -
         \norm{ \Y^2 - \X^1 }
 >
        2 + 2 - 2 = 2\, , 
\end{align*}
contrary to the fact that $\diam(A)\le 2$. The result follows. 
\qed\end{proof}
\bigbreak

\begin{lemma}\label{stejna}
We have
$\|\X^1-\X^2\|=\|\Y^1-\Y^2\|$.
\end{lemma}
\begin{proof}
Consider any point $V
        =(v_1,v_2)
\in(D(\Y^1,2)\cap D(\Y^2,2))\setminus D(0,R)^o$.
Note that $v_2\in [\x^2_2,\x^1_2]$ and $v_1 \geq \x^1_1$.
We claim that $A\subset D(V,2)$. To see the claim, suppose that
there exists $T
        =(t_1,t_2)
\in A\setminus D(V,2)$.
We can use Lemma~\ref{conv} in the same way as before to deduce that
$t_1\in [\y^2_1,\x^1_1]$.
Also note that $\Y^1, \Y^2 \in D(V,2)$ while $T\notin D(V,2)$.
We have either $\conv(\{\X^1,T\})\cap\conv(\{\Y^2,V\})\neq\emptyset$ or $\conv(\{\X^2,T\})\cap\conv(\{\Y^1,V\})\neq\emptyset$ and, in a similar way as in Lemma~\ref{dlouha}, the triangle inequality implies that either $\|T-\X^1\|>2$ or $\|T-\X^2\|>2$, which is a contradiction.

Now consider the set 
\[
\hat{A}=\conv(A)\cup((D(\Y^1,2)\cap D(\Y^2,2))\setminus D(0,R)^o)  
\]
and notice that, by the previous consideration, it holds $((\hat{A}+ (\tau,0))\cap D)\in\M$ for any $\tau\in\R$. We distinguish two cases. 

Suppose first that $\|\Y^1 -\Y^2 \| - \|\X^1 - \X^2\| > 2\epsilon$, for some $\epsilon >0$. 
Then there exists $\delta>0$ such that     
\[
\lambdatwo((\hat{A}+(\delta,0))\cap D)
\geq\lambdatwo(A)+(\|\Y^1-\Y^2\|-\epsilon)\delta-(\|\X^1-\X^2\|+\epsilon)\delta > \lambdatwo(A)  \, .
\]
In other words, the set $((\hat{A}+(\delta,0))\cap D) \in \M$ has larger measure than $A$ and contradicts Running assumption~\ref{St_a}. 

Similarly, if $\|\X^1 -\X^2 \| - \|\Y^1 - \Y^2\| > 2\epsilon$, for some $\epsilon>0$, then there exists $\delta>0$ such that 
\[
\lambdatwo((\hat{A}-(\delta,0))\cap D)\geq\lambdatwo(A)-(\|\Y^1-\Y^2\|+\epsilon)\delta+(\|\X^1-\X^2\|-\epsilon)\delta> \lambdatwo(A)\, , 
\]
which contradicts Running assumption~\ref{St_a}.  
We conclude  $\|\X^1-\X^2\|=\|\Y^1-\Y^2\|$, as desired.
\qed\end{proof}

By Lemma~\ref{stejna} and Running assumption~\ref{St_a_2}, it follows that the four points $\X^1,\X^2,\Y^1,\Y^2$ form the vertices of a rectangle with center $\BS$, where $\BS$ is defined in Notation~\ref{not_1}. 
In the remaining part of this section, we use the following.
                \NOOPxbeforeThm

\begin{lemma}\label{delkaA}
We have
\[
\left(\frac12\|\X^1-\X^2\|\right)^2=\frac{-R^4+16R^2}{8(R^2+2)} \;.
\]
\end{lemma}
\begin{figure}
\begin{center}
		\includegraphics[scale=1.0]{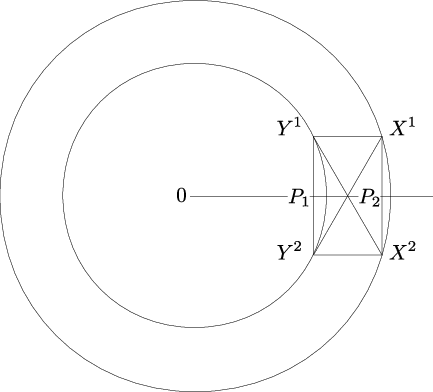}
\end{center}	
	\caption{Situation in Lemma~\ref{delkaA}.}
	\label{fig:Pythagoras}
\end{figure}
\begin{proof}
Let us denote by $P_1$ the intersection of the segment $Y^1Y^2$ with the $x$-axis and by $P_2$ the intersection of the segment $X^1X^2$ with the $x$-axis. See Figure~\ref{fig:Pythagoras}. Let us write $\xi:=\|0-P_1\|$ and $\zeta:=\|P_1-P_2\|$.
Using the Pythagorean theorem for triangles $Y^1Y^2X^2$, $0P_1Y^1$, and $0P_2X^2$ together with Lemmas~\ref{dlouha} and~\ref{stejna}.
we obtain
\begin{align}
\label{Pyth1} \|\X^1-\X^2\|^2+\zeta^2&=2^2\;,\\
\label{Pyth2} \xi^2+(\tfrac12\|\X^1-\X^2\|)^2&=2^2\;,\\
\label{Pyth3} (\xi+\zeta)^2+(\tfrac12\|\X^1-\X^2\|)^2&=R^2\;.
\end{align}
By squaring~\eqref{Pyth3}, we get
\begin{equation}\label{eq:expandme}
4\xi^2\zeta^2=(R^2-(\tfrac12\|\X^1-\X^2\|)^2-\xi^2-\zeta^2)^2\;.
\end{equation}
We now expand~\eqref{eq:expandme}. Then we use~\eqref{Pyth1} and~\eqref{Pyth2} to eliminate $\zeta^2$ and $\xi^2$, respectively. This leads to
\[0=R^4+2R^2\|\X^1-\X^2\|^2-16R^2+4\|\X^1-\X^2\|^2\;,\]
and the lemma follows.
\end{proof}
		
The proof of Theorem~\ref{isodiam_annulus} is almost complete. 
\begin{proofxx}
We can assume that the set $A$ satisfies all the assertions derived above. We use Notation~\ref{not_1}. Further, we write 
\begin{align*}
K&=\{(v_1,v_2)\in D(\X^1,2)\cap D(\X^2,2)\cap D(\Y^1,2)\cap D(\Y^2,2); v_2\notin[\x^2_2,\x^1_2]\}\;,\\
O&=\{(x,y)\in D;\ y\in[\x^2_2,\x^1_2]\text{ and }x>0\}\;.
\end{align*}
Last, let $L=\{(\s,\tau)\in\R^2;\ \tau\in\R\}$ be the line orthogonal to $\conv(\{0,(R,0)\})$ that contains $\BS$.

Clearly we have $A\subset O\cup K$ and $K\subset D$. Moreover, notice that $K$ is symmetric with respect to $L$. Now we apply the classical Steiner symmetrization (see~\cite[p.~87]{Evans_Gariepy}) with respect to $L$ to $K\cap A$ to obtain a new set $\TA$.  This guarantees that  $\lambdatwo(\TA)=\lambdatwo(A\cap K)$. Moreover, since $\diam(K\cap A)\leq2$, we have $\diam(\TA)\leq2$. Since $K$ is symmetric with respect to $L$ we obtain $\TA\subset K\subset D$. Now observe that for every $V\in K$ we have $\diam(O\cup\{V\})\leq2$. Since $\TA\subset K$ and $\diam(\TA)\leq2$ we have $\diam(\TA\cup O)\leq2$, and it follows that $\TA\cup O$ is $D$-maximal. Since $A=A^{st}$, $\TA$ is also symmetric with respect to the line $\{(\tau,0);\ \tau\in\R\}$. Thus $\TA$ is symmetric with respect to $\BS$. Since $\diam(\TA)\leq2$ we have $\TA\subset D(\BS,1)$. Clearly, $O\subset D(\BS,1)$. Hence $D(\BS,1)\cap D$ is $D$-maximal. Therefore, all the previously derived properties apply also to the set $D(\BS,1)\cap D$. Note that the set $D(\BS,1)\cap D$ is of the form as in Figure~\ref{fig:Annulus}. So, all it takes now is to calculate the area of $D(\BS,1)\cap D$. 
\begin{fact}\label{fact:areaA}
Recall that the constatnt $a$ was defined in the statement of Theorem~\ref{isodiam_annulus}. We have
\begin{align*}
\lambdatwo(D(\BS,1)\cap D)=R^2\arcsin\left(\frac{\ba}{R}\right)-4\arcsin\left(\frac{\ba}{2}\right)+2\arccos(\ba)
\;.
\end{align*}
\end{fact}
\begin{proof}
Denoting by $X_*^1,X_*^2,Y_*^1,Y_*^2$ the ``corners'' of $D(\BS,1)\cap D$ (as in Notation~\ref{not_1}), Lemma~\ref{delkaA} gives 
\[ 
\left(\frac12\left(\|X_*^1-X_*^2\|\right)\right)^2=\frac{-R^4+16R^2}{8(R^2+2)}=a^2 \;.\]
\begin{figure}
\begin{center}
		\includegraphics[scale=1.0]{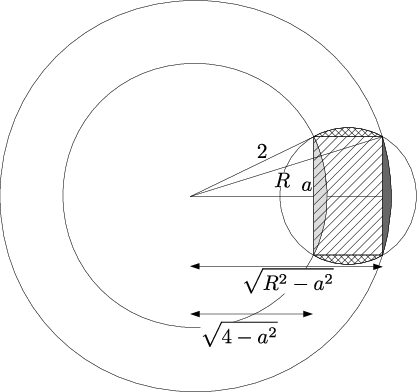}
\end{center}	
	\caption{Situation in Fact~\ref{fact:areaA}.}
	\label{fig:AreaD}
\end{figure}To calculate the area of $D(\BS,1)\cap D$,
we compare it to the
rectangle $X_*^1Y_*^1Y_*^2X_*^2$ whose area is $2\ba(\sqrt{R^2 - \ba^2}-\sqrt{4 - \ba^2})$.
The rectangle is depicted with a line-hatching in Figure~\ref{fig:AreaD}.
The area depicted in dark grey is $R^2\arcsin\left(\frac{\ba}{R}\right)-\ba\sqrt{R^2 - \ba^2}$. The area of each of the parts depicted with a cross-hatching is $\arccos(\ba)-{\tfrac 12}\ba(\sqrt{R^2 - \ba^2}-\sqrt{4 - \ba^2})$.
Last, we need to subtract the part depicted in light grey, area of which is $4\arcsin\left(\frac{\ba}{2}\right)-\ba\sqrt{4 - \ba^2}$.
Therefore,
\[
\lambdatwo(D(\BS,1)\cap D)=R^2\arcsin\left(\frac{\ba}{R}\right)-4\arcsin\left(\frac{\ba}{2}\right)+2\arccos(\ba).\]
This finishes the proof of Fact~\ref{fact:areaA} and hence also of Theorem~\ref{isodiam_annulus}.
\qed  % Two proofs end here actually, as the text says.
\end{proof}
\end{proofxx}

\section{Concluding remarks}\label{conclusion}

\subsection{Problem~\ref{gen_isodiam} in higher dimension}

When $d\le 4$, we conjecture that a measurable subset $A\subset \mathbb{R}^d$ for which $\mathcal{G}_A$ is $K_k$-free satisfies $\lambda_d(A)\le (k-1)\cdot \omega_d$. Perhaps rather surprisingly, for $d\ge 5$ the answer is different. This can already be seen from the case of $K_3$-free sets. 
Notice that the radius of a $2$-ball that circumscribes an equilateral triangle whose sides are equal to $2$ is equal to $2/\sqrt{3}$. Now it is easy to see that any $d$-ball of radius $2/\sqrt{3}$ (for arbitrary $d\ge 2$) is $K_3$-free. The volume of such a $d$-ball is equal to $\left(2/\sqrt{3}\right)^d\cdot\omega_d$. 
Now it is not difficult to see that $2\omega_d < \left(2/\sqrt{3}\right)^d\cdot\omega_d$, when $d\ge 5$, and therefore the optimal $K_3$-free set is not a disjoint union of two unit balls that are at sufficiently large distance. We do not have a conjecture for the optimal set in higher dimensions.  

Notice that in Theorem~\ref{2d_mantel} we found the best possible upper bound on the measure of the ``edge set'' of a $2$-dimensional $K_3$-free set by exploiting the fact that the optimal configuration maximizing the measure of the ``vertex set'' is a disjoint union of two unit balls that are at distance at least $2$. From the discussion of the previous paragraph, we know that a disjoint union of two unit balls is not the optimal configuration in higher dimensions and therefore, and perhaps rather surprisingly, there might be instances for which the set $A$ which maximizes $e(\mathcal{G}_A)$, in the setting of  Problem~\ref{gen_isodiam}, is different from the set $A$ which maximizes $\lambda_d(A)$.

\subsection{Isodiametric inequality for annuli in higher dimension}
Central to our proof of Theorem~\ref{2d_mantel} was the isodiametric inequality for annuli, Theorem~\ref{isodiam_annulus}. However, Theorem~\ref{isodiam_annulus} was stated only in dimension $2$. It would be of interest to extend the result to other dimensions.
\NOOPxbeforeThm
\begin{problem}\label{annulus_problem}
Let $R\in (2,4]$ and consider the set 
$D= D(0,R)\setminus D(0,2)^o \subset \mathbb{R}^d$. 
Suppose that $A$ is a compact subset of $D$ such that $\diam(A)\le 2$. What is a sharp upper bound on the $\lambda_d$-measure of $A$?
\end{problem}

Problem~\ref{annulus_problem} can in turn be employed to find the maximum $\lambda_d$-measure of a $K_3$-free set.  
It appears that an approach that is similar to the approach of Section~\ref{vasek} can be employed in higher dimensions. The only subtlety is in the calculations involving the higher-dimensional analogue of the circular symmetrization and at the moment it is unclear how to proceed.  
We leave this as an open problem for the reader and we hope that we will be able to report on that matter in the future. 

\subsection{The structure of large-distance graphs}
What is the structure of large-distance graphs? There are many ways to formalize this problem, but let us put forward a particular one.
\begin{problem}\label{prob:dens}
Given a dimension $d\ge 1$, describe the set 
$$\mathcal{S}_d:=\left\{\left(\lambda_{v(H)\cdot d}(\mathcal{G}_A\langle H\rangle)\right)_{\mbox{$H$ graph}}\::\: \mbox{$A\subset\mathbb{R}^d$ of finite measure} \right\}\subset \mathbb{R}^{\mbox{graphs}}\;.$$
\end{problem}
Recall that in Remark~\ref{remIsodiametric} and directly beneath it we pointed out that Problem~\ref{prob:dens} is not scale-free. Even for $d=1$, the answer does not seem obvious.

Let us note that the counterpart of Problem~\ref{prob:dens} is one of the cornerstones of the theory of limits of dense graph sequences. More specifically, Lov\'asz and Szegedy~\cite{lovsze} characterized the set
\begin{equation}\label{eq:T}
\mathcal{T}:=\left\{\left(t(H,W)\right)_{\mbox{$H$ graph}}\::\: \mbox{$W$ graphon}\right\}\subset \mathbb{R}^{\mbox{graphs}}\;.
\end{equation}

Note that if $W$ is a graphon associated to a set $A\subset \mathbb{R}^d$ (as we did in Section~\ref{moonMoser}), then
$$\lambda_{v(H)\cdot d}(\mathcal{G}_A\langle H\rangle)=\lambda_d(A)^{v(H)}\cdot t(H,W)\quad\mbox{for each $H$}.$$
Thus, for each $d$, the set $\mathcal{S}_d$ is a subset of a suitable transformed (by the factors $\lambda_d(A)^{v(H)}$) set $\mathcal{T}$. It seems that when $d$ is large, this set inclusion is actually not far from an equality.

\section*{Acknowledgements}
We thank the anonymous referee for his or her comments.


\begin{thebibliography}{99}
\bibitem{Aigner} M. Aigner, \textit{Tur\'an's graph theorem}, The American Mathematical Monthly \textbf{102}  (1995) 808--816.

\bibitem{Bollobas}  B. Bollob\'as,  \textit{Measure graphs}, J. London Math. Soc. (1980) 401--407.

\bibitem{BollobasExtremal} B. Bollob\'as, \textit{Extremal graph theory}, Dover Publications, Inc., Mineola, NY, 2004.


\bibitem{Bonnesen_Fenchel} T. Bonnesen, W. Fenchel, \textit{Theory of Convex Bodies}, BCS Associates, Moscow,
Idaho, USA, 1987


\bibitem{Burago_Zalgaller} Yu. D. Burago, V. A. Zalgaller, \textit{Geometric Inequalities}, Springer-Verlag, 1988. 

\bibitem{Evans_Gariepy} L. C. Evans, R. F. Gariepy, \textit{Measure Theory and Fine Properties of Functions}, CRC Press, Revised Edition, 2015.

\bibitem{DHKMPVEurocomb} M. Dole\v zal, J. Hladk\'y, J. Kol\'a\v r, T. Mitsis, C. Pelekis, V. Vlas\'ak, \textit{A Tur\'{a}n-type theorem for large-distance graphs in Euclidean
   spaces, and related isodiametric problems}, Acta Math. Univ. Comenian. (N.S.) \textbf{88} 
(2019) 625--629.

\bibitem{Goodman} A. W. Goodman, \textit{On sets of acquaintances and strangers at any party}, Amer. Math. Monthly \textbf{66} 
(1959) 778--783.

\bibitem{KahleFigueroa} M. Kahle, F.  Martinez-Figueroa, \textit{The chromatic number of random Borsuk graphs}, To appear in Random Structures \& Algorithms, arXiv:1901.08488.  

\bibitem{katonaone} G. O. H. Katona, \textit{Continuous versions of some extremal hypergraph problems}, Combinatorics, Keszthely (Hungary), 1976, Coll. Math. Soc. J. Bolyai 18 (Math. Soc. J. Bolyai, Budapest, 1978) 653--678.
 
\bibitem{katonatwo} G. O. H. Katona, \textit{Continuous versions of some extremal hypergraph problems II}, Acta Math. Acad. Sci. Hungar. \textbf{35} (1980) 067--077.

\bibitem{Katona} G. O. H. Katona, \textit{Tur\'an's graph theorem, measures and probability theory},  Number theory, analysis, and combinatorics, De Gruyter Proc. Math., De Gruyter, Berlin, (2014). pp. 167--176.

\bibitem{snow} J. Koll\'{a}r, L. R\'{o}nyai, T. Szab\'{o}, \textit{Norm-graphs and bipartite {T}ur\'{a}n numbers}, Combinatorica \textbf{16} (1996), no. 3, 399--406.

\bibitem{Lovasz} L. Lov\'asz, \textit{Large networks and graph limits}, Vol. 60 of  American Mathematical Society Colloquium Publications. American Mathematical Society, Providence, RI, 2012. 

\bibitem{lovsze} L. Lov\'asz, B. Szegedy, \textit{Limits of dense graph sequences}, J. Combin. Theory Ser. B \textbf{96} (2006), no. 6, 933--957. 

\bibitem{Mantel} W. Mantel, \textit{Problem 28}, Wiskundige Opgaven \textbf{10} (1907) 60--61.

\bibitem{Matousek} J. Matou\v{s}ek, \textit{Using the Borsuk-Ulam theorem}, Springer Berlin Heidelberg, 2008.

\bibitem{pelekis} C. Pelekis, \textit{A generalized isodiametric problem}, Geombinatorics \textbf{25}  (2016), no. 4, 151--167. 

\bibitem{Polya}  G. P\'olya, \textit{Sur la symm\'etrisation circulaire}, C. R. Acad. Sci. Paris, S\'er. I Math.
(1) \textbf{230} (1950) 25--27

\bibitem{ShabanovRaigorodskii} L. E. Shabanov, A. M. Raigorodskii, \textit{Tur\'an type results for distance graphs}, Discrete and Computational Geometry (2016) \textbf{56}: 814 -- 832. 

\bibitem{SimonovitsSos} M. Simonovits, V. T. S\'os, \textit{Ramsey-Tur\'an theory}, Discrete Mathematics \textbf{229} (2001) 293--340. 

\bibitem{turan} P. Tur\'an, \textit{On an extremal problem in graph theory}, Matematikai \'es Fizikai Lapok (in Hungarian), (1941) \textbf{48}: 436--452

\end{thebibliography}
\end{document}